\documentclass[oneside,english]{amsart}

\usepackage{amsthm}
\usepackage{amstext}
\usepackage{amssymb}
\usepackage[unicode=true]
 {hyperref}
\usepackage{breakurl}
\usepackage{url}
\usepackage{enumitem}

\usepackage{geometry}
\makeatletter

\numberwithin{equation}{section}
\numberwithin{figure}{section}
\theoremstyle{plain}
\newtheorem{thm}{Theorem}

  \theoremstyle{definition}

  \theoremstyle{definition}
  \newtheorem*{defn*}{Definition}

  \theoremstyle{plain}
  \newtheorem*{thm*}{Theorem}

  \theoremstyle{plain}

  \theoremstyle{remark}
  \newtheorem*{rem*}{Remark}

  \theoremstyle{definition}
  \newtheorem*{example*}{Example}

  \theoremstyle{plain}

  \theoremstyle{plain}

  \theoremstyle{remark}

  \theoremstyle{plain}

\usepackage{amsmath}
\allowdisplaybreaks
\usepackage{longtable}

\newcommand{\field}[1]{\mathbb{#1}}

\newcommand{\Q}{\field{Q}}
\newcommand{\QQ}{\field{Q}}
\newcommand{\CC}{\field{C}}
\newcommand{\Fq}{\field{F}_q}
\newcommand{\FFq}{\field{F}_q}

\newcommand{\Z}{\field{Z}}

\AtBeginDocument{
  
}
\def\blfootnote{\xdef\@thefnmark{}\@footnotetext}

\makeatother

\begin{document}

\title{Characterizations of the $d$th-power residue matrices over finite fields}

\author{Evan P. Dummit}
\address{Evan P. Dummit, Arizona State University, School of Mathematical and Statistical Sciences, P.O. Box 871804, Tempe AZ 85287-1804}
\email{evan.dummit@asu.edu}

\begin{abstract}
In a recent paper of the author with D. Dummit and H. Kisilevsky, we constructed a collection of matrices defined by quadratic residue symbols, termed ``quadratic residue matrices'',  associated to the splitting behavior of prime ideals in a composite of quadratic extensions of $\QQ$, and proved a simple criterion characterizing such matrices.  We then analyzed the analogous classes of matrices constructed from the cubic and quartic residue symbols for a set of prime ideals of $\QQ(\sqrt{-3})$ and $\QQ(i)$, respectively.  In this paper, the goal is to construct and study the finite-field analogues of these residue matrices, the  ``$d$th-power residue matrices'', using the general $d$th-power residue symbol over a finite field.
\end{abstract}

\maketitle

\section{The $d$th-Power Residue Matrices}

Our goal is to study the appropriate analogue of the residue matrices constructed in \cite{D-D-K} in the finite-field setting.

Let $q$ be a prime power and $d$ be a positive integer with $d$ dividing $q-1$, and let $\Fq$ denote the finite field with $q$ elements.  We begin by recalling the standard definition and some basic properties of the $d$th-power residue symbol for polynomials in $\Fq[t]$ .  

\medskip
\noindent
\begin{defn*} If $P$ is a monic irreducible polynomial over $\Fq$ and $a \in \Fq[t]$ is relatively prime to $P$, the $d$th-power residue symbol $\left( \dfrac{a}{P} \right)_d$ 
is defined to be the unique $d$th root of unity in $\Fq$ with
$$
\left(\dfrac{a}{P}\right)_d \equiv a^{ (\left| P \right| -1)/d }\,\, (\mathrm{mod}\,\,P) 
$$
where $\left| P \right|$ denotes the norm of $P$, defined as $q^{\deg(P)}$, the cardinality of $\Fq[t] / (P)$.
\end{defn*}

We remark here that the $d$th power residue map $\left( \dfrac{\cdot}{P} \right)_d$ is a surjective homomorphism from the multiplicative group of nonzero residue classes modulo $P$ to the group of $d$th roots of unity in $\Fq$.

It will be convenient instead to consider the $d$th-power residue symbol as taking values in $\CC$: to this end, choose a fixed isomorphism $\varphi$ of the $d$th roots of unity in $\Fq$ with the complex $d$th roots of unity.

\noindent
\begin{defn*} If $P$ is a monic irreducible polynomial over $\Fq$ and $a \in \Fq[t]$, we define the modified $d$th-power residue symbol $\left[ \dfrac{a}{P} \right] _d$ to be the complex root of unity with $\left[ \dfrac{a}{P} \right] _d = \varphi \left( \left( \dfrac{a}{P} \right)_d \right)$.
\end{defn*}

We remark here (and will justify later) that the resulting class of matrices is independent of the isomorphism $\varphi$: any other isomorphism will produce the same class of matrices.

\noindent
\begin{defn*}
Let $d$ be a positive integer.  A ``cyclotomic sign matrix of $d$th roots of unity'' is an $n \times n$ matrix whose diagonal entries are all 0 and whose off-diagonal entries are all complex $d$th roots of unity.
\end{defn*}

With the correct class of matrices in hand, we can now define the $d$th-power residue matrices.

\begin{defn*}
Let $q$ be a prime power and $d$ be an integer dividing $q-1$.  The ``$d$th-power residue'' matrix associated to the monic irreducible polynomials $P_{1}$, $P_{2}$, $\dots$, $P_{n}$ in $\Fq[t]$ is the $n\times n$ matrix whose $(i,j)$-entry is the $d$th power residue symbol $\left[ \dfrac{P_{i}}{P_{j}} \right] _d$.
\end{defn*}

Notice that the $d$th-power residue matrices are cyclotomic sign matrices of $d$th roots of unity.  We would like to characterize, for a given $d$ and $q$, which cyclotomic sign matrices of $d$th roots of unity actually arise as the $d$th-power residue matrix associated to some set of monic irreducible polynomials over $\Fq$.  We should naturally expect $d$th-power reciprocity to impose some conditions.  

Over $\FFq[t]$ the $d$th-power reciprocity law is as follows (cf. Theorem 3.3 of \cite{R}): for any monic irreducible polynomials $P$ and $Q$ in $\Fq[t]$,
$$\left(\dfrac{P}{Q}\right)_{d} = (-1)^{(q-1)\deg(P)\deg(Q)/d} \left(\dfrac{Q}{P}\right)_{d}$$
and for the modified residue symbols the statement is the same [except with square brackets].

\section{Characterizations of the $d$th-Power Residue Matrices}

Observe that if $(q-1)/d$ is even then the $d$th-power reciprocity law is symmetric, and thus all of the $d$th-power matrices are symmetric.  The converse is also true:

\begin{thm} Let $q$ be a prime power and $d$ be an integer dividing $q-1$ with $(q-1)/d$ even. If $M$ is an $n \times n$ cyclotomic sign matrix of $d$th roots of unity, then the following are equivalent:

\begin{enumerate} [label=(\alph*)]
\item The matrix $M$ is symmetric.

\item The matrix $M$ is the $d$th-power residue matrix associated to
distinct monic irreducible polynomials $P_{1}$, $P_{2}$, $\dots$, $P_{n}$ in $\Fq[t]$.
\end{enumerate}
\end{thm}

\begin{proof}
(a) implies (b): 
We inductively construct monic irreducible polynomials $P_{1}, \dots , P_{n}$ for which
$M$ is the $d$th-power residue matrix.  For the base case, let $P_{1}$ be any
monic irreducible polynomial of positive degree. For the inductive step, suppose that
$P_{1},$ ... , $P_{k}$ are monic irreducible polynomials such that $\left[\dfrac{P_{i}}{P_{j}}\right]_{d}=m_{i,j}$
for $1\leq i,j\leq k$. For each $1\leq j\leq k$, choose a nonzero
residue class $u_{j}$ modulo $P_{j}$ such that $\left[\dfrac{u_{j}}{P_{j}}\right]_{d}=m_{k+1,j}$.
By the Chinese Remainder Theorem and Kornblum's function-field analogue of Dirichlet's Theorem on primes
in arithmetic progression (cf. Theorem 4.7 of [R]) we may choose a monic irreducible polynomial $P_{k+1}$ satisfying the congruences
$P_{k+1}\equiv u_{j}$ (mod $P_{j}$) for all $1 \leq j \leq k$. By construction, we have $\left[\dfrac{P_{k+1}}{P_{j}}\right]_{d}=m_{k+1,j}$
for all $1\leq j\leq k$, and $d$th-power reciprocity along with the form of $M$ ensures that also $\left[\dfrac{P_{i}}{P_{k+1}}\right]_{d}=m_{i,k+1}$ for all $1\leq i\leq k$ is satisfied. Thus, $M$ is the $d$th-power residue matrix
associated to $P_{1}, \dots , P_{n}$, as claimed.

(b) implies (a): This follows immediately from $d$th-power reciprocity, since $$\left(\dfrac{P_i}{P_j}\right)_{d} = \left(\dfrac{P_j}{P_i}\right)_{d}$$ for all pairs $(i,j)$ with $i \neq j$.  
\end{proof}

When $(q-1)/d$ is odd, $d$th-power reciprocity takes a form quite similar to quadratic reciprocity over $\Q$, with polynomials of even and odd degree behaving like rational primes congruent to $1$ and $3$ (mod $4$), respectively: if either $P$ or $Q$ has even degree, then $\left(\frac{P}{Q}\right)_{d} = \left(\frac{Q}{P}\right)_{d}$, and if both have odd degree then $\left(\frac{P}{Q}\right)_{d} = -\left(\frac{Q}{P}\right)_{d}$.

Observe that the property of whether a cyclotomic sign matrix is a $d$th-power residue matrix is invariant under conjugation by a permutation matrix (simply permute the underlying polynomials accordingly).  If we reorder the polynomials so that the first $s$ have odd degree and the remaining $n-s$ have even degree, then by $d$th-power reciprocity the associated $d$th-power residue matrix $M$ has the form
$$
\begin{pmatrix}
A & B \\
B^t & S
\end{pmatrix}
$$
where $A$ is an $s \times s$ skew-symmetric cyclotomic sign matrix of $d$th roots of unity, $S$ is an $(n-s) \times (n-s)$ symmetric cyclotomic sign matrix of $d$th roots of unity, and $B$ is an $s\times(n-s)$ matrix all of whose entries are $d$th roots of unity.  (Here $B^t$ denotes the transpose of $B$.)

We now show that every matrix having the form above is a $d$th-power residue matrix when $(q-1)/d$ is odd, and give an additional characterization:

\begin{thm}
 Let $q$ be a prime power and $d$ be an integer dividing $q-1$ with $(q-1)/d$ odd.  If $M$ is an $n \times n$ cyclotomic sign matrix of $d$th roots of unity, then the following are equivalent:
\begin{enumerate} [label=(\alph*)]
\item There exists an integer $s$ with $1 \le s \le n$ such that the matrix $M$ can be 
conjugated by a permutation matrix into a block matrix of the form
$$
\begin{pmatrix}
A & B \\
B^t & S
\end{pmatrix}
$$
where $A$ is an $s \times s$ skew-symmetric cyclotomic sign matrix of $d$th roots of unity, $S$ is an $(n-s) \times (n-s)$ symmetric cyclotomic sign matrix of $d$th roots of unity, and $B$ is an $s\times(n-s)$ matrix all of whose entries are $d$th roots of unity.  (Here $B^t$ denotes the transpose of $B$.)

\item The matrix $M$ is the $d$th-power residue matrix associated to
a set of distinct monic irreducible polynomials $P_{1}, P_{2}, \dots, P_{n}$ in $\Fq[t]$.

\item If $M = (m_{j,k})$, then $m_{j,k} = \pm m_{k,j}$ for all $j,k$ with $1 \le j,k \le n$, and there exists
an integer $s$ with $1 \leq s \leq n$ such that the diagonal entries of $M \overline{M}$
consist of $s$ occurrences of $n+1-2s$ and $n-s$ occurrences of $n-1$.

\end{enumerate}
\end{thm}

\begin{proof}
(a) implies (b): Follows by the same proof as in Theorem 1, except we additionally impose the condition that the degree of the polynomial $P_{k+1}$ is odd if $k \leq s$ or even if $k > s$, in order to obtain the correct entries below the diagonal.

(b) implies (c): Suppose that $M$ is the $d$th-power residue matrix associated
to the distinct monic irreducible polynomials $P_{1}, \dots , P_{n}$.
The first part of the criterion in (c) follows immediately from $d$th-power reciprocity.

For the second part, rearrange the polynomials, if necessary, so that the first $s$ have odd degree and the remaining $n-s$ have even degree.  Note also that for any $d$th root of unity $r$ in $\Fq$, $\varphi(r^{-1}) = \varphi(r)^{-1} = \overline{\varphi(r)}$.

For $1 \leq j \leq s$, the $j$th diagonal element of $M\overline{M}$ is
$$
(M\overline{M})_{j,j}= \sum_{k=1}^{n} \left[ \dfrac{P_{j}}{P_{k}} \right]_{d} \overline{ \left[ \dfrac{P_{k}}{P_{j}} \right]_{d} } = \sum_{k=1}^{n} \varphi \left(\left( \dfrac{P_{j}}{P_{k}} \right)_{d} \left( \dfrac{P_{k}}{P_{j}} \right)_{d}^{-1}  \right)= n+1-2s 
$$
since by $d$th-power reciprocity the first $s$ terms are $-1$ (except
for the $j$th, which is 0), and the other $n-s$ terms are $+1$.

For $s+1 \leq j \leq n$, the $j$th diagonal element of $M\overline{M}$ is 
$$
(M\overline{M})_{j,j}= \sum_{k=1}^{n} \left[ \dfrac{P_{j}}{P_{k}} \right]_{d} \overline{ \left[ \dfrac{P_{k}}{P_{j}} \right]_{d} } = \sum_{k=1}^{n} \varphi \left(\left( \dfrac{P_{j}}{P_{k}} \right)_{d} \left( \dfrac{P_{k}}{P_{j}} \right)_{d}^{-1}  \right) = n-1 
$$
since by $d$th-power reciprocity all terms are $+1$ (except for the $j$th, which is 0), proving (c).

(c) implies (a): Suppose that $m_{j,k}=\pm m_{k,j}$ for each pair
$(j,k)$, and that the diagonal entries of the matrix $M \overline{M}$
consist of $s$ occurrences of $n+1-2s$ and $n-s$ occurrences of
$n-1$. 

Whenever $j\neq k$, by the assumptions that $m_{j,k}=\pm m_{k,j}$ and that the $m_{j,k}$ are $d$th roots of unity, we see that $m_{j,k} \overline{m_{k,j}}$ is either $+1$ (when $m_{j,k}= m_{k,j}$) or $-1$ (when $m_{j,k}= -m_{k,j}$).

By conjugating $M$ by an appropriate permutation matrix we may place the $s$ occurrences of $n+1-2s$
in the first $s$ rows of $M \overline{M}$. 
For $s < j \leq n$, we have 
$$
(M \overline{M})_{j,j}= \sum_{k=1}^{n} m_{j,k} \overline{m_{k,j}} = n-1 ,
$$
but since there are only $n-1$ nonzero terms in the sum, we necessarily have 
$m_{j,k} \overline{m_{k,j}}=1$ for each $j \neq k$, and hence $m_{j,k}=m_{k,j}$ for all
$1\leq k \leq n$ and $s < j \leq n$. 

For $1 \leq j \leq s$, we have
$$
(M \overline{M})_{j,j}= \sum_{k=1}^{n} m_{j,k} \overline{m_{k,j}} = n+1-2\cdot\#\{1 \leq k \leq s\,:\, m_{j,k} \overline{m_{k,j}} = -1 \}
$$
since $m_{j,k} \overline{m_{k,j}} = +1$ whenever $j>s$ and $m_{j,k}\overline{m_{k,j}}$
can only be $1$ or $-1$. But now since there at most $s$ terms
in the count, and $(M\overline{M})_{j,j}=n+1-2s$, we see that $m_{j,k} \overline{m_{k,j}}=-1$ and hence that $m_{j,k}=-m_{k,j}$ for $1\leq k \leq s$.  Thus $M$ has the form in (a), completing the proof.   
\end{proof}

\begin{rem*}
Observe that both condition (a) of Theorem 1, and conditions (a) and (c) of Theorem 2, are wholly independent of the choice of isomorphism $\varphi$ between the $d$th roots of unity in $\Fq$ and the complex $d$th roots of unity, and therefore we see that the classes of $d$th-power residue matrices are the same no matter which $\varphi$ is used.
\end{rem*}

In a similar manner to the way the quadratic, cubic, and quartic residue matrices classify certain types of decomposition configurations over number fields (cf. \cite{D-K}), the fact that not every $n\times n$ cyclotomic sign matrix of $d$th roots of unity arises as a $d$th-power residue matrix has implications for the possible decomposition configurations for primes in abelian extensions of $\Fq(t)$ with Galois group $(\Z / d \Z)^n$.

\section*{Acknowledgements}
The author would like to thank David Dummit for his many helpful comments during the preparation of this paper.

\bibliographystyle{plain}
\bibliography{qcqmatrices}

%
%
%
%
%
%
%
%
%

\end{document}